\documentclass[a4paper,12pt,final]{amsart}
\usepackage{times,a4wide,mathrsfs,amssymb,dsfont,wasysym}

\newcommand{\C}{\mathbb{C}}

\newcommand{\QQ}{\mathbb{Q}}

\newcommand{\PP}{\mathbb{P}}
\newcommand{\LL}{\mathcal{L}}
\newcommand{\LLL}{\mathbb{L}}

\newcommand{\OO}{\mathcal O}
\newcommand{\Ss}{\mathcal S}

\newcommand{\FF}{\mathcal F}
\newcommand{\XX}{\mathcal X}

\newcommand{\CC}{\mathcal C}

\newcommand{\MM}{\mathcal M}

\newcommand{\pic}{\hbox{Pic}}

\newcommand{\wt}{\widetilde}
\newcommand{\ima}{\hbox{Im}}
\newcommand{\rom}{\romannumeral}

\newcommand{\one}{\mathds{1}}

\newtheorem{theorem}{Theorem}[section]

\newtheorem{lemma}[theorem]{Lemma}

\newtheorem{corollary}[theorem]{Corollary}
\newtheorem{proposition}[theorem]{Proposition}
\newtheorem{conjecture}[theorem]{Conjecture}
\newtheorem{problem}[theorem]{Problem}
\newtheorem{nonumbering}{Theorem}

\newtheorem{nonumberingc}{Corollary}

\newtheorem{convention}{Conventions}

\theoremstyle{definition}
\newtheorem{remark}[theorem]{Remark}
\newtheorem{definition}[theorem]{Definition}

\newtheorem{notation}[theorem]{Notation}

\newtheorem{nonumberingt}{Acknowledgements}

\begin{document}
\author[Robert Laterveer]
{Robert Laterveer}

\address{Institut de Recherche Math\'ematique Avanc\'ee,
CNRS -- Universit\'e 
de Strasbourg,\
7 Rue Ren\'e Des\-car\-tes, 67084 Strasbourg CEDEX,
FRANCE.}
\email{robert.laterveer@math.unistra.fr}

\title{On the Chow ring of Fano varieties of type $S2$}

\begin{abstract} We show that certain Fano eightfolds (obtained as hyperplane sections of an orthogonal Grassmannian, and studied by Ito--Miura--Okawa--Ueda and by Fatighenti--Mongardi) have a multiplicative Chow--K\"unneth decomposition. As a corollary, the Chow ring of these eightfolds behaves like that of K3 surfaces.
\end{abstract}



\maketitle

\section{Introduction}

   For a smooth projective variety $X$ over $\C$, let $A^i(X)=CH^i(X)_{\QQ}$ denote the Chow group of codimension $i$ algebraic cycles modulo rational equivalence with $\QQ$-coefficients, and let $A^i_{hom}(X)$ denote the subgroup of homologically trivial cycles. Intersection product defines a ring structure on $A^\ast(X)=\oplus_i A^i(X)$ \cite{F}. In the case of K3 surfaces, this ring structure has a remarkable property:

\begin{theorem}[Beauville--Voisin \cite{BV}]\label{K3} Let $S$ be a K3 surface. Let $D_i, D_i^\prime\in A^1(S)$ be a finite number of divisors. Then
  \[ \sum_i D_i\cdot D_i^\prime=0\ \ \ \hbox{in}\ A^2(S)_{}\ \Leftrightarrow\ \sum_i D_i\cdot D_i^\prime=0\ \ \ \hbox{in}\ H^4(S,\QQ)\ .\]
  \end{theorem}

In the wake of this result (combined with results concerning the Chow ring of abelian varieties), Beauville has asked for which varieties the Chow ring behaves similarly to Theorem \ref{K3}. (This is the problem of determining which varieties verify the ``splitting property'' of \cite{Beau3}; this circle of ideas has notably led to the famous ``Beauville--Voisin conjecture'' concerning the Chow ring of hyperk\"ahler varieties \cite{Beau3}, \cite{V9}.)
We briefly state this problem here as follows:

\begin{problem}[Beauville \cite{Beau3}]\label{prob} Find a class $\CC$ of varieties (containing $K3$ surfaces, abelian varieties and hyperk\"ahler varieties, and stable under taking products), such that for any $X\in\CC$, the Chow ring of $X$ admits a multiplicative bigrading $A^\ast_{(\ast)}(X)$, with
  \[ A^i(X)=\bigoplus_{j= 0}^i A^i_{(j)}(X)\ \ \ \hbox{for\ all\ }i\ .\] 
This bigrading should split the conjectural Bloch--Beilinson filtration, in particular 
  \[ A^i_{hom}(X)= \bigoplus_{j\ge 1} A^i_{(j)}(X)\ .\]
  \end{problem}
  
 This question is hard to answer in practice, since we do not have the Bloch--Beilinson filtration at our disposal. 
 
An interesting alternative approach to Problem \ref{prob} (as well as a reinterpretation of Theorem \ref{K3}) is provided by the concept of {\em multiplicative Chow--K\"unneth decomposition\/}, giving rise to unconditional constructions of a bigraded ring structure on the Chow ring of many particular varieties \cite{SV}, \cite{V6}, \cite{SV2}, \cite{FTV}, \cite{FV}, \cite{LV}. Thus one is led to the following problem, which is more concrete than Problem \ref{prob}:

\begin{problem}[Shen--Vial \cite{SV}]\label{prob2} Describe the class $\CC^\prime$ of varieties having a multiplicative Chow--K\"unneth decomposition.
\end{problem}

To relate this to Problem \ref{prob}, one might naively conjecture that $\CC=\CC^\prime$ (and that for any $X\in\CC=\CC^\prime$, the induced bigradings on $A^\ast(X)$ coincide). As a partial answer towards Problem \ref{prob2}, I have proposed the following:

\begin{conjecture}\label{conj} Let $X$ be a smooth projective Fano variety of K3 type (i.e. $\dim X=2m$ and the Hodge numbers $h^{p,q}(X)$ are $0$ for all $p\not=q$ except for $h^{m-1,m+1}(X)=h^{m+1,m-1}(X)=1$). Then $X$ has a multiplicative Chow--K\"unneth decomposition.
\end{conjecture}

This conjecture is answered positively in a few scattered cases \cite{d3}, \cite{Ver}, \cite{B1B2}, \cite{FLV2}.
The aim of the present note is to provide some more evidence for Conjecture \ref{conj}, by considering certain Fano eightfolds studied by Ito--Miura--Okawa--Ueda \cite{IMOU} and Fatighenti--Mongardi \cite{FM}. Following \cite{FM}, we say that a variety {\em of type $S2$\/} is a smooth divisor in a certain very ample linear system $\LL$ on the orthogonal Grassmannian $\operatorname{OGr}(3,8)$. Varieties of type $S2$ are Fano eightfolds of K3 type (cf. subsection \ref{ss:8} below). The main result of this note is a verification of Conjecture \ref{conj} for varieties of type $S2$:

\begin{nonumbering}[=Theorem \ref{main}] Let $X$ be a variety of type $S2$.
Then $X$ has a multiplicative Chow--K\"unneth decomposition.
  \end{nonumbering}
  
This is proven by first showing that for a general variety $X$ of type $S2$,  certain genus $7$ K3 surfaces that are naturally associated to $X$ are also related to $X$ on the level of Chow motives (Theorem \ref{main0}).
 
  As a nice bonus, the theorem implies that the Chow ring of these Fano eightfolds behaves like that of K3 surfaces:
  
 \begin{nonumberingc}[=Corollary \ref{cor}]  Let $X$ be a variety of type $S2$. Let $R^5(X)\subset A^5(X)$ denote the subgroup
   \[ R^5(X):= \langle A^1(X)\cdot A^4(X), \ A^2(X)\cdot A^3(X) ,\ c_5(T_X),\ \ima\bigl( A^\ast(\operatorname{OGr}(3,8))\to A^\ast(X)\bigr) \rangle\ .\]
   Then the cycle class map induces an injection
   \[ R^5(X)\ \hookrightarrow\ H^{10}(X,\QQ)\cong \QQ^4\ .\]
  \end{nonumberingc} 
   
  Since $A^5_{hom}(X)$ is infinite-dimensional (for $X$ general, there is an isomorphism $A^5_{hom}(X)\cong A^2_{hom}(S)$, where $S$ is an associated K3 surface, cf. Theorem \ref{main0}), this is just as remarkable as Theorem \ref{K3}. 
   
It would be interesting to test Conjecture \ref{conj} for the other Fano varieties of K3 type exhibited in \cite{FM} and \cite{BFM}. I hope to return to this in the near future.

\noindent
    
\vskip0.6cm

\begin{convention} In this note, the word {\sl variety\/} will refer to a reduced irreducible scheme of finite type over the field of complex numbers $\C$. 
{\bf All Chow groups will be with $\QQ$-coefficients, unless indicated otherwise:} For a variety $X$, we will write $A_j(X):=CH_j(X)_{\QQ}$ for the Chow group of dimension $j$ cycles on $X$ with rational coefficients.
For $X$ smooth of dimension $n$, the notations $A_j(X)$ and $A^{n-j}(X)$ will be used interchangeably. 
The notation
$A^j_{hom}(X)$ will be used to indicate the subgroups of 
homologically trivial  cycles.

We will write  $\MM_{{\rm rat}}$ for the contravariant category of Chow motives (i.e., pure motives as in \cite{Sc}, \cite{MNP}).
\end{convention}

\section{Preliminaries}

\subsection{MCK decomposition}
\label{ssmck}

\begin{definition}[Murre \cite{Mur}] Let $X$ be a smooth projective variety of dimension $n$. We say that $X$ has a {\em CK decomposition\/} if there exists a decomposition of the diagonal
   \[ \Delta_X= \Pi^0_X+ \Pi^1_X+\cdots +\Pi^{2n}_X\ \ \ \hbox{in}\ A^n(X\times X)\ ,\]
  such that the $\Pi^i_X$ are mutually orthogonal idempotents and $(\Pi^i_X)_\ast H^\ast(X)= H^i(X)$.
  
  (NB: ``CK decomposition'' is shorthand for ``Chow--K\"unneth decomposition''.)
\end{definition}

\begin{remark} The existence of a CK decomposition for any smooth projective variety is part of Murre's conjectures \cite{Mur}, \cite{J2}. 
\end{remark}

\begin{definition}[Shen--Vial \cite{SV}] Let $X$ be a smooth projective variety of dimension $n$. Let $\Delta^X_{sm}\in A^{2n}(X\times X\times X)$ be the class of the small diagonal
  \[ \Delta_X^{sm}:=\bigl\{ (x,x,x)\ \vert\ x\in X\bigr\}\ \subset\ X\times X\times X\ .\]
  An MCK decomposition is a CK decomposition $\{\Pi^i_X\}$ of $X$ that is {\em multiplicative\/}, i.e. it satisfies
  \[ \Pi^k_X\circ \Delta_X^{sm}\circ (\Pi^i_X\times \Pi^j_X)=0\ \ \ \hbox{in}\ A^{2n}(X\times X\times X)\ \ \ \hbox{for\ all\ }i+j\not=k\ .\]
  
 (NB: ``MCK decomposition'' is shorthand for ``multiplicative Chow--K\"unneth decomposition''.) 
  \end{definition}
  
  \begin{remark} The small diagonal (seen as a correspondence from $X\times X$ to $X$) induces the {\em multiplication morphism\/}
    \[ \Delta_X^{sm}\colon\ \  h(X)\otimes h(X)\ \to\ h(X)\ \ \ \hbox{in}\ \MM_{\rm rat}\ .\]
 Suppose $X$ has a CK decomposition
  \[ h(X)=\bigoplus_{i=0}^{2n} h^i(X)\ \ \ \hbox{in}\ \MM_{\rm rat}\ .\]
  By definition, this decomposition is multiplicative if for any $i,j$ the composition
  \[ h^i(X)\otimes h^j(X)\ \to\ h(X)\otimes h(X)\ \xrightarrow{\Delta^X_{sm}}\ h(X)\ \ \ \hbox{in}\ \MM_{\rm rat}\]
  factors through $h^{i+j}(X)$.
  It follows that if $X$ has an MCK decomposition, then setting
    \[ A^i_{(j)}(X):= (\Pi_X^{2i-j})_\ast A^i(X) \ ,\]
    one obtains a bigraded ring structure on the Chow ring: that is, the intersection product sends 
    $A^i_{(j)}(X)\otimes A^{i^\prime}_{(j^\prime)}(X) $ to  $A^{i+i^\prime}_{(j+j^\prime)}(X)$.
      
  The property of having an MCK decomposition is severely restrictive, and is closely related to Beauville's ``weak splitting property'' \cite{Beau3}. For more ample discussion, and examples of varieties with an MCK decomposition, we refer to \cite[Section 8]{SV} and also \cite{V6}, \cite{SV2}, \cite{FTV}, \cite{LV}, \cite{FLV2}.
    \end{remark}

\subsection{Varieties of type $S2$}
\label{ss:8}

\begin{notation}\label{not} Let $\operatorname{OGr}(3,8)$ denote the orthogonal Grassmannian of $3$-dimensional isotropic subspaces of an $8$-dimensional vector space equipped with a bilinear form. As explained in \cite[Section 2]{IMOU} and \cite[Section 3.7]{FM}, there are morphisms
  \[ \begin{array}[c]{ccccc}
       &&\operatorname{OGr}(3,8)&&\\
       &{\scriptstyle {}^{p_1}}\swarrow&&\searrow{\scriptstyle {}^{p_2}}&\\
       &&&&\\
       F_1&&&&F_2\\
       \end{array}\]
       where $F_1\cong F_2$ are the two connected components of $\operatorname{OGr}(4,8)$, and the $p_i$ are $\PP^3$-fibrations.
     The Picard group of $\operatorname{OGr}(3,8)$ has rank $2$ and is generated by the pullbacks $(p_i)^\ast \pic(F_i)$, $i=1,2$.
     The line bundle 
     \[ \LL:= (p_1)^\ast \OO_{F_1}(1)\otimes (p_2)^\ast \OO_{F_2}(1)     \ \ \ \in\ \pic(  \operatorname{OGr}(3,8))\]
     is very ample, and there are natural isomorphisms
     \begin{equation}\label{iso} H^0(\operatorname{OGr}(3,8),\LL)\cong H^0(F_1, (p_1)_\ast \LL)\cong H^0(F_2, (p_2)_\ast \LL)\ .\end{equation}
      \end{notation}

\begin{definition} A variety {\em of type $S2$\/} is by definition a smooth hypersurface in the linear system $\vert\LL\vert$ on $\operatorname{OGr}(3,8)$ (cf. Notation \ref{not}).
\end{definition}

\begin{proposition}\label{8} Let $X$ be a a variety of type $S2$. Then $X$ is a Fano variety of dimension $8$.
The Hodge diamond of $X$ is
\[ \begin{array}[c]{ccccccccc}
	&&&&1&&&&\\
	&&&  &2&  &&&\\
	&&&&3&&&&\\
	&&&&4&&&&\\	
	0& \dots&\dots&1&24&1&\dots&\dots&0\\
	&&&&4&&&&\\	&&&&3&&&&\\
	&&&  &2&  &&&\\
	&&&&1&&&&\\
	\end{array}\]	
(where all empty entries are $0$).	
\end{proposition}

\begin{proof} This is \cite[Section 3.7]{FM}. The Hodge numbers are computed in \cite[Proposition A.1.1]{Fa}.

\end{proof}

\begin{remark} Varieties of type $S2$ are part of a long list of Fano varieties of K3 type given in \cite{FM}. 
\end{remark}

\begin{proposition}[]\label{assk3} Given $s\in H^0(\operatorname{OGr}(3,8),\LL)$ a general section, let $s_1, s_2$ be the induced sections of $(p_1)_\ast\LL$ resp. $(p_2)_\ast\LL$ under the isomorphism (\ref{iso}). Let $S_i\subset F_i$ denote the zero loci of $s_i$ ($i=1,2$). Then $S_1$ and $S_2$ are 
K3 surfaces of genus $7$. Moreover,

\noindent
(\rom1) $S_1$ and $S_2$ are L-equivalent: one has equality in the Grothendieck ring of varieties
  \[  (S_1-S_2)\cdot \LLL^3=0\ \ \ \hbox{in}\ K_0(\hbox{Var})\ ;\]
  
\noindent
(\rom2) $S_1$ and $S_2$ are derived equivalent (i.e. their derived categories of coherent sheaves are isomorphic);

\noindent
(\rom3) for $s$ very general, $S_1$ and $S_2$ are not isomorphic.
\end{proposition}

\begin{proof} Except for (\rom2), this is contained in \cite{IMOU}. The hypothesis that $s$ be {\em very general\/} in (\rom3) is made to ensure that the $S_i$ have Picard number $1$ (cf. \cite[Section 4]{IMOU}).

Statement (\rom2) (which we merely state for illustration, and will not use below) is announced in \cite{FM} and proven in \cite{BFM}.
\end{proof}

\begin{remark} Proposition \ref{assk3} (\rom2) implies in particular, via \cite{Huy}, that $S_1$ and $S_2$ have isomorphic Chow motives. This will be proven directly below (Corollary \ref{K3iso}), without appealing to the derived equivalence.

\end{remark}

\begin{definition} Let $X$ be a variety of type $S2$, and assume that the sections $s_i$ as in Proposition \ref{assk3} define smooth surfaces $S_i$. We say that $S_1$ and $S_2$ are {\em associated\/} to $X$.
\end{definition}

\subsection{The Franchetta property}
\label{ss:gfc}

\begin{definition}[\cite{FLV}] Let $\XX\to B$ be a smooth projective morphism. We say that $\XX\to B$ has {\em the Franchetta property\/} if the following holds: any 
$\Gamma\in A^\ast(\XX)$ which is fibrewise homologically trivial is fibrewise rationally trivial.
\end{definition}

\begin{theorem}\label{gfc} Let
      $B$ be the Zariski open in $\PP H^0(\operatorname{OGr}(3,8),\LL)$ parametrizing smooth dimensionally transverse hypersurfaces, and let $B_0\subset B$ such that for each fibre $X_b=V(s_b)$ with $b\in B_0$ the zero locus $S_b:=V((pr_1)_\ast(s))\subset F_1$ is a smooth K3 surface. Let
      \[ \Ss\ \to\ B_0 \]
      denote the universal family of sections of type $S_b$.
      
   The families 
   \[ \Ss\ \to \ B_0\ ,\ \ \Ss\times_{B_0}\Ss\ \to\ B_0 \]
   have the Franchetta property. 
\end{theorem}

\begin{proof} This is not a surprising result: indeed, $\Ss\to B_0$ contains the general K3 surface of genus $7$ (i.e. there is a dominant morphism $B_0\to \FF_7$ to the moduli stack of genus $7$ K3 surfaces), and the Franchetta property for $\Ss\to\FF_7$ and for $\Ss\times_{\FF_7}\Ss$ were already
proven in \cite{PSY} resp. in \cite{FLV}. However, in \cite{PSY}, \cite{FLV} the standard Mukai parametrization of genus $7$ K3 surfaces was used, which is different from the parametrization $B_0$ used here, and so we need to do some extra work to prove Theorem \ref{gfc}.

The argument for $\Ss\to B_0$ is similar to that of \cite{PSY}. We use the following:

\begin{lemma}[\cite{IMOU}]\label{lemfib} Let $\pi\colon X\to F_1$ denote the restriction of $p_1\colon \operatorname{OGr}(3,8)\to F_1$ to $X$.
 The morphism $\pi\colon X\to F_1$ is a (Zariski locally trivial) $\PP^3$-bundle over $S_1\subset F_1$, and a (Zariski locally trivial) $\PP^2$-bundle over $U:= F_1\setminus S_1$.
 \end{lemma}
 
 \begin{proof} This is \cite[Lemma 2.1]{IMOU}.
 \end{proof} 

Let 
  \[\bar{B}:=\PP H^0(\operatorname{OGr}(3,8),\LL)\ ,\] 
  and let us consider the universal family 
  \[\bar{\Ss}\ \to\ \bar{B}\]
 of possibly singular sections. There is an inclusion as a Zariski open $B_0\subset\bar{B}$. It follows from Lemma \ref{lemfib} that $\bar{\Ss}$
is a $\PP^r$-fibration over $F_1$ (indeed, given a point $y\in F_1$ let $O_y:=(p_1)^{-1}\cong\PP^3\subset  \operatorname{OGr}(3,8)$ denote the fibre over $y$. Since $\LL$ is base-point free, there exists a section of $\LL$ not containing the whole fibre $O_y$, hence there is a surface $S_b\subset F_1$ avoiding the point $y$: every point $y\in Y$ imposes one condition on $\bar{B}$).

Reasoning as in \cite[Lemma 2.1]{PSY}, this implies that 
  \[ \ima\Bigl(  A^\ast(\bar{\Ss})\to A^\ast(S_b)\Bigr) = \ima\Bigl( A^\ast(F_1)\to A^\ast(S_b)\Bigr)\ \ \ \forall b\in B_0\ .\]
  Since $A^2(F_1)\cong\QQ$ is generated by intersections of divisors, this settles the Franchetta property for $\Ss\to B_0$.

Next, we claim that the family $\Ss\to B_0$ verifies property $(\ast_2)$ of \cite[Definition 5.6]{FLV}. This claim, combined with \cite[Proposition 5.7]{FLV} and the Franchetta property for $\Ss\to B_0$ implies that
\[ \ima\Bigl(  A^\ast(\bar{\Ss}\times_{\bar{B}}\bar{\Ss})\to A^\ast(S_b\times S_b)\Bigr) =  \langle A^1(S_b),\Delta_{S_b}\rangle   \ \ \ \forall b\in B_0\ .\]
The right-hand side is known to inject into cohomology \cite[Proposition 2.2]{V9}, and so we are done.

To prove the claim, we reason as above: given two different points $y_1, y_2\in F_1$, let $O_{y_1}, O_{y_2}\subset  \operatorname{OGr}(3,8) $ denote the fibres. Given the definition of $\LL$, one readily finds that restriction induces a surjection
 \[ H^0( \operatorname{OGr}(3,8), \LL)\ \twoheadrightarrow\ H^0(O_{y_1},\LL\vert_{O_{y_1}})\oplus H^0(O_{y_2},\LL\vert_{O_{y_2}})\ ,\]
  i.e. two different points $y_1,y_2
\in F_1$ impose 2 independent conditions on $\bar{B}$. This proves the claim.
\end{proof}

\section{An isomorphism of motives}

\begin{theorem}\label{main0} Let $X\subset\operatorname{OGr}(3,8)$ be a variety of type $S2$, and assume that $X$ has an associated K3 surface $S$. There is an isomorphism of motives
  \[  h(X)\cong h(S)(-3)\oplus \bigoplus \one(\ast)\ \ \ \hbox{in}\ \MM_{\rm rat}\ .\]
  (In particular, one has $A^i_{hom}(X)=0$ for all $i\not=5$.)
\end{theorem}

\begin{proof} Without loss of generality we may assume $S=S_1$ with respect to the notation introduced above.
Let $\pi\colon X\to F_1$ and $U:= F_1\setminus S_1$ be as in Lemma \ref{lemfib}, i.e. $\pi$ is generically a $\PP^2$-fibration but degenerates to a $\PP^3$-fibration over $S_1$.
Then one has an isomorphism of motives
  \begin{equation}\label{ji}
    h(X)\cong  \bigoplus_{j=0}^2 h(F_1)(-j)\oplus h(S)(-3)\ \ \ \hbox{in}\ \MM_{\rm rat}\ .\end{equation}
  Since $F_1\cong\operatorname{Spin}(8)$ is a $6$-dimensional quadric (cf. \cite[Remark 2.4]{IMOU}), the motive $h(F_1)$ is isomorphic to a sum of twisted Lefschetz motives $\oplus \one(\ast)$, and so the theorem follows from \eqref{ji}.
    
 The isomorphism \eqref{ji} can be proven using Voevodsky motives (as I did in \cite[Proof of Theorem 2.1, equation (4)]{S6}, where the situation is completely similar), but also follows directly by applying \cite[Corollary 3.2]{Ji}.

    \end{proof}

\begin{corollary}\label{K3iso} Let $X$ be a variety of type $S2$, and assume $X$ has associated K3 surfaces $S_1, S_2$. Then
  \[ h(S_1)\ \cong\ h(S_2)\ \ \ \hbox{in}\ \MM_{\rm rat}\ .\]
  \end{corollary}
  
  \begin{proof} Theorem \ref{main0} implies that there are isomorphisms of motives
   \[   \Gamma_i\colon\ \ h(X)\ \xrightarrow{\cong}\ t(S_i)(-3)\oplus \bigoplus \one(\ast)\ \ \ \hbox{in}\ \MM_{\rm rat} \ \ \ \ (i=1,2)\ ,\]   
and so one gets an isomorphism
 \[ t(S_1)\oplus \bigoplus \one(\ast) \cong t(S_2)\oplus \bigoplus \one(\ast)\ \ \ \hbox{in}\ \MM_{\rm rat} \ .\]
Taking Chow groups, this implies that there is a correspondence-induced isomorphism
  \[  A^\ast(t(S_1))=A^2_{hom}(S_1)\cong A^2_{hom}(S_2)=A^\ast(t(S_2))\ .\]
The Bloch--Srinivas argument (cf. for instance \cite[Lemma 1.1]{Huy}) then implies 
 that $t(S_1)\cong t(S_2)$. Since both $S_i$ are K3 surfaces and so have the same Betti numbers, one can conclude that $h(S_1)\cong h(S_2)$.
        \end{proof}

\section{MCK for varieties of type $S2$}

\begin{theorem}\label{main} Let $X$ be a variety of type $S2$. Then $X$ has a multiplicative Chow--K\"unneth decomposition.
  
  The Chern classes $c_j(T_X)$, as well as cycles in the image of the restriction $A^\ast(\operatorname{OGr}(3,8))\to A^\ast(X)$,
  are in $A^\ast_{(0)}(X)$.
    \end{theorem}
 
 \begin{proof} Let
     \[ \XX\ \to\ B\  \]
     denote the universal family of smooth hyperplane sections of $ \operatorname{OGr}(3,8)$ (i.e. $B$ is a Zariski open in the space of sections $\PP H^0(\operatorname{OGr}(3,8),\LL)$, where notation is as above). Let $B_0\subset B$ and $ \Ss\to B_0$ be as in subsection \ref{ss:gfc}.
     In view of a standard spread argument (cf. \cite[Lemma 3.2]{Vo}), it suffices to prove the theorem for $X_b$ with $b\in B_0$.
     
     Theorem \ref{main0} gives us an isomorphism
   \begin{equation}\label{isomot} h(X_b)\cong h(S_b)(-3)\oplus \bigoplus \one(\ast)\ \ \ \hbox{in}\ \MM_{\rm rat}\ ,\end{equation}
   for each $b\in B_0$. Even better, this isomorphism exists {\em universally\/}. Indeed, the maps
   \[ \Psi_b\colon\ \ h(S_b)(-3)\ \to\ h(X_b)  \]
   that enter into the isomorphism (\ref{isomot}) clearly exist universally (they are defined in terms of pullback and pushforward). A Hilbert schemes argument as in \cite[Proposition 3.7]{V0} (cf. also \cite[Proposition 2.11]{Lat} for the precise form used here) then implies that the $\Gamma_b$ also exist universally, i.e. there exists $\Gamma\in A^{5}(\XX\times_{B_0}\Ss)\oplus A^\ast(\XX)$ inducing fibrewise isomorphisms
   \begin{equation}\label{iso1} \Gamma\vert_b\colon\ \  h(X_b)\cong h(S_b)(-3)\oplus \bigoplus \one(\ast)\ \ \ \hbox{in}\ \MM_{\rm rat}\ ,\end{equation}
   for each $b\in B_0$.
   
   One can readily construct a universal CK decomposition for $X$, i.e. there exist cycles $\pi^i_{\XX}\in A^8(\XX\times_B \XX)$ such that the restriction
     \[ \pi^i_{X_b}:=  \pi^i_{\XX}\vert_b\ \ \ \in A^8(X_b\times X_b) \]
     defines a CK decomposition for $X_b$ for each $b\in B$. (This is a standard construction, cf. for instance \cite[Lemma 3.6]{V0}. In brief, one observes that for any $i<8$, $H^i(X_b)\cong H^i(\operatorname{OGr}(3,8))$ is algebraic, and so
     $\pi^i_{X_b}$ is of the form $\sum_k a_{ik}^\vee\times a_{ik}$, where the $a_{ik}$ are a basis for $H^i(\operatorname{OGr}(3,8))$ and the $a_{ik}^\vee$ form a dual basis for $H^{16-i}(X_b)\cong H^{18-i}(\operatorname{OGr}(3,8))$. One then defines $\pi^i_\XX$, $i< 8$ by restricting appropriate cycles in $ \operatorname{OGr}(3,8)\times  \operatorname{OGr}(3,8)\times B$. The cycles $\pi^i_\XX$, $i>8$ are defined as the transpose of $\pi_\XX^{16-i}$, and the
     remaining cycle $\pi^8_\XX$ is defined as the difference $\Delta_\XX-\sum_{j\not=8} \pi^j_{\XX}$.)
     
     It remains to check that this is an MCK decomposition, i.e. one needs to check that the relative correspondence
     \[  \Phi_{ijk}:=  \pi^k_\XX\circ \Delta^{sm}_\XX\circ \bigl((pr_{13})^\ast(\pi^i_\XX)\cdot (pr_{24})^\ast(\pi^j_\XX)\bigr)\ \ \ \in A^{16}(\XX\times_B \XX\times_B \XX)\ ,\ \ \ i+j\not=k \]
     is fibrewise equal to $0$. (Here $\Delta^{sm}_\XX\in A^{16}(\XX\times_B \XX\times_B \XX)$ denotes the relative small diagonal.)    
The assumption $i+j\not= k$ implies that $\Phi_{ijk}$ is fibrewise homologically trivial. Thus, the image
     \[ (\Gamma,\Gamma,\Gamma)_\ast (\Phi_{ijk})\ \ \in A^7(\Ss\times_{B_0}\Ss\times_{B_0}\Ss)\oplus A^\ast(\Ss\times_{B_0}\Ss)\oplus A^\ast(\Ss)\oplus A^\ast(B_0)\]
     is also fibrewise homologically trivial. But then the Franchetta property for $\Ss\times_{B_0}\Ss$ and for $\Ss$ (Theorem \ref{gfc}), plus the fact that $A^7((S_b)^3)=0$ for dimension reasons, implies that
      \[      (\Gamma\vert_b,\Gamma\vert_b,\Gamma\vert_b)_\ast (\Phi_{ijk}\vert_b) =    (\Gamma,\Gamma,\Gamma)_\ast (\Phi_{ijk})\vert_b =0\ \ \ \hbox{for\ all\ }b\in B_0\ .\]
    Since $ (\Gamma\vert_b,\Gamma\vert_b,\Gamma\vert_b)_\ast $ is injective (cf. isomorphism (\ref{iso1})), this proves that 
     \[ \Phi_{ijk}\vert_b=0\ \ \ \hbox{in}\  A^{16}(\XX\times_B \XX\times_B \XX)\ ,\ \  \ \hbox{for\ all\ }b\in B_0\ \hbox{and}\  i+j\not=k\ . \]      
    This proves the first part of the theorem.
    
     For the second part, one observes that the above argument (passing from the family $\XX$ to the family $\Ss$) shows that the family $\XX\to B$ (and also the family $\XX\times_B \XX\to B$) has the Franchetta property (that is, any $\Gamma\in A^\ast(\XX)$ that is fibrewise homologically trivial is fibrewise zero). 
    It follows that
     \begin{equation}\label{new}  \ima\bigl(  A^j(\XX)\to A^j(X_b)\bigr)\ \subset\ A^j_{(0)}(X)\ ,\end{equation}
     as one sees by applying the Franchetta property to
     \[    (\pi^i_{X_b})_\ast (a\vert_{X_b})=     \bigl( (\pi^i_\XX)_\ast (a)\bigr)\vert_{X_b}\ \ \ , i\not=2j\ ,\]
     where $a\in A^j(\XX)$.
     The inclusion \eqref{new} applies to the Chern classes $c_j(T_{X_b})=c_j(T_{\XX/B})\vert_{X_b}$ and also to cycles in $\ima \bigl(  A^\ast(\operatorname{OGr}(3,8))\to A^\ast(X_b)\bigr)$. This proves the second part of the theorem.     
      \end{proof}

 \begin{corollary}\label{cor} Let $X$ be an eightfold as in Theorem \ref{main}. Let $R^5(X)\subset A^5(X)$ denote the subgroup
   \[ R^5(X):= \langle A^1(X)\cdot A^4(X), \ A^2(X)\cdot A^3(X) ,\ c_5(T_X),\ \ima\bigl( A^\ast(\operatorname{OGr}(3,8))\to A^\ast(X)\bigr) \rangle\ .\]
   Then the cycle class map induces an injection
   \[ R^5(X)\ \hookrightarrow\ H^{10}(X,\QQ)\cong \QQ^4\ .\]
   \end{corollary}
   
   \begin{proof} Since $A^i_{hom}(X)=0$ for $i\not=5$, we have $A^i(X)=A^i_{(0)}(X)$ for $i\not=5$. Combined with Theorem \ref{main}, this implies that $R^5(X)\subset A^5_{(0)}(X)$. It only remains to check that the cycle class map induces an injection
   \[ A^5_{(0)}(X)\ \hookrightarrow\ H^{10}(X,\QQ) \ .\]
   To this end, we observe that (by construction) the correspondence $\pi^{10}_X$ is supported on $V\times W\subset X\times X$, where $V$ resp. $W$ are subvarieties of dimension $5$ resp. $3$. As in \cite{BS}, the action of $\pi^{10}_X$ on $A^5(X)$ factors over $A^0(\wt{W})$ (where $\wt{W}\to W$ is a resolution of singularities). In particular, the action of $\pi^{10}_X$ on $A^5_{hom}(X)$  factors over  $A^0_{hom}(\wt{W})=0$ and so is zero. But the action of  $\pi^{10}_X$ on $A^5_{(0)}(X)$  is the identity, and so
    \[     A^5_{(0)}(X)\cap A^5_{hom}(X)=0\ .\]
       \end{proof}

\vskip1cm
\begin{nonumberingt} Thanks to the reviewer for helpful comments.
Thanks to Kai for beautifully playing his "Duo de Printemps" \quarternote\twonotes\fullnote.
\end{nonumberingt}

\end{document}